\documentclass[11pt,a4paper]{article}
\usepackage{amsmath}
\usepackage{amsthm}
\usepackage{amsfonts}
\usepackage{latexsym}
\usepackage{graphicx}  % to deal with graphics
 \usepackage{hyperref} 

\newcounter{alphthm}
\newtheorem{prop}{Proposition}[section]
\newtheorem{thm}{Theorem}[section]

\newtheorem{cor}{Corollary}[section]

\theoremstyle{definition}
\newtheorem{definition}{Definition}[section]
\newtheorem{rem}{Remark}[section]
\newtheorem{ex}{Example}

\newcommand{\be}{\begin{equation}}
\newcommand{\ee}{\end{equation}}
\newcommand{\ben}{\begin{enumerate}}
\newcommand{\een}{\end{enumerate}}

\begin{document}
\title{T--Extended Weakly Contractive, Kannan, and Geraghty Mappings: Fixed Points, Equivalences, and Rates}

\author{
  Fatemeh Fogh\thanks{e-mail: ffogh2021@fau.edu}\\
  Department of Mathematical Sciences, Florida Atlantic University, USA
  \and
  Sara Behnamian\thanks{e-mail: sara.behnamian@sund.ku.dk}\\
  Globe Institute, University of Copenhagen, Copenhagen, 1350, Denmark
}

\maketitle

\begin{abstract}
We develop a unified $T$--extended framework for weakly contractive, weakly Kannan, and Geraghty classes of selfmaps $S$ on a metric space $(X,d)$, where distances are measured on the auxiliary image via $d(T\cdot,T\cdot)$ and the dynamics is governed by $TS$. Under the standard hypotheses on the auxiliary map $T$ (continuity, injectivity, subsequential convergence), we prove fixed--point theorems and Picard convergence for each class. Our main novelty is twofold: first, we show that the $T$--extended weakly contractive class coincides with the $T$--extended Geraghty class, and that the $T$--extended weakly Kannan class coincides with the $T$--extended Kannan--Geraghty class; second, we identify the precise mechanism behind these coincidences by transporting the problem to the induced map $F:T(X)\to T(X)$, $F(Tx)=TSx$, where the ``extended'' properties are exactly the classical ones (same control). We also give a $\Delta$--type ratio criterion on $T(X)$ and quantitative Picard rates, and we include examples (e.g.\ Volterra smoothing) that highlight the role of the auxiliary map. All results extend verbatim to rectangular (Branciari) metric spaces.
\end{abstract}


\def\thefootnote{ \ }
\footnotetext{{\em} $2010$ Mathematics Subject Classification: 47H10; 47H09}

\section{Introduction}
The Banach fixed point theorem (1922) provides existence, uniqueness, and Picard convergence for contractive selfmaps on complete metric spaces. Kannan's theorem \cite{Blair11} furnishes an independent principle: if
\[
d(fx,fy)\le \frac{\alpha}{2}\bigl[d(x,fx)+d(y,fy)\bigr]\qquad(\alpha\in[0,1)),
\]
then $f$ has a unique fixed point and the Picard iterates converge. Variable--coefficient relaxations, such as weakly contractive mappings in the sense of Dugundji--Granas \cite{Blair1} and weakly Kannan mappings \cite{000012}, preserve these conclusions in complete spaces. Geraghty \cite{Blair14} replaced the constant contraction factor by a control $\beta\in\Gamma$, obtaining
\[
d(fx,fy)\le \beta\!\bigl(d(x,y)\bigr)\,d(x,y),
\]
with $\beta(t_n)\to1\Rightarrow t_n\to0$, again ensuring a unique fixed point (see also \cite{Blair8}). These ideas sit alongside other generalizations (e.g.\ \(F\)-contractions \cite{Wardowski2012}, \(\alpha\)\(-\)\(\psi\) contractions \cite{Samet2012}, and simulation functions \cite{Khojasteh2015}).

A different but fruitful viewpoint is to transport contractive structure through an auxiliary map $T:X\to X$ and to work with $d(T\cdot,T\cdot)$ and the composition $TS$; under mild hypotheses on $T$ (continuity, injectivity, subsequential convergence), one obtains fixed points for suitable $T$--dependent contractive forms \cite{a2,a1}. The same auxiliary--map philosophy has recently been developed in the non--self setting via best proximity points \cite{FoghBehnamian2026BPP}, where $S$--proximal Geraghty and Kannan--Geraghty classes are introduced for pairs $(A,B)$ of subsets. The present paper places the weakly contractive, weakly Kannan, and Geraghty inequalities uniformly on the $T$--image and studies the resulting \emph{$T$--extended} classes \cite{Fogh2019KG, a3}.

\medskip
\noindent\textbf{What is new here.}
Our contributions are conceptual and structural. First, we prove fixed--point theorems (with Picard convergence when $T$ is sequentially convergent) for the $T$--extended weakly contractive, weakly Kannan, and Geraghty classes (Theorems~\ref{thm:TwC}, \ref{thm:TwK}, and \ref{thm:TG}). Second, we show that two pairs of classes actually coincide on the $T$--image: the $T$--extended weakly contractive and $T$--extended Geraghty classes are equivalent, and likewise the $T$--extended weakly Kannan and $T$--extended Kannan--Geraghty classes (Theorem~\ref{thm:equivPairs}). The mechanism is transparent: if $F:T(X)\to T(X)$ is defined by $F(Tx)=TSx$, then the $T$--extended properties of $S$ on $X$ are precisely the classical properties of $F$ on $(T(X),d)$ with the same control (Proposition~\ref{prop:image-equivalence}); in particular, injectivity of $T$ is essential for uniqueness (Remark~\ref{rem:injective}). Third, we provide a $\Delta$--type ratio criterion formulated on $T(X)$ and derive quantitative Picard rates (Corollary~\ref{cor:rates}). Finally, we verify that all arguments pass unchanged to rectangular (Branciari) metric spaces \cite{bran}.

\medskip
\noindent\textbf{Organization.}
Section~\ref{def:TwK}--\ref{thm:TG} introduces the $T$--extended classes and proves the corresponding fixed--point theorems. Section~\ref{subsec:equiv} identifies the induced map $F$ on $T(X)$ and establishes the pairwise equivalences of classes together with consequences for uniqueness and convergence; it also contains examples that illustrate how $T$ can convert nonexpansive behavior into contractive behavior (Volterra smoothing). The rectangular--metric analogues are recorded at the end of the section.

\section{Preliminaries}
A map $Q:X\to X$ is called \emph{sequentially convergent} if, for every sequence $\{y_n\}\subset X$, the convergence of $\{Qy_n\}$ implies that $\{y_n\}$ itself converges in $X$. It is called \emph{subsequentially convergent} if, for every sequence $\{y_n\}\subset X$, the convergence of $\{Qy_n\}$ implies that $\{y_n\}$ has a convergent subsequence in $X$; see \cite{a1}.

A map $f:X\to X$ is \emph{weakly contractive} if there exists $\overline{\alpha}:X\times X\to[0,1)$ such that
\[
d(fx,fy)\le \overline{\alpha}(x,y)\,d(x,y)\qquad(\forall\,x,y\in X),
\]
and $\sup\{\overline{\alpha}(x,y):\, a\le d(x,y)\le b\}<1$ for every $0<a\le b$ \cite{Blair1}. In a complete metric space, every weakly contractive map has a unique fixed point and its Picard iterates converge to it \cite{Blair1}. In particular, every $k$--contraction with $k\in[0,1)$ is weakly contractive by taking $\overline{\alpha}\equiv k$.

Kannan's theorem states that if $(X,d)$ is complete and $f:X\to X$ satisfies
\[
d(fx,fy)\le\frac{\alpha}{2}\bigl[d(x,fx)+d(y,fy)\bigr]\qquad(\forall\,x,y\in X),
\]
for some $\alpha\in[0,1)$, then $f$ has a unique fixed point and the Picard iterates converge to it \cite{Blair11}. A \emph{weakly Kannan} mapping is defined by the same inequality with $\alpha$ replaced by $\overline{\alpha}(x,y)\in[0,1)$ satisfying the same boundedness condition on annuli; a corresponding fixed--point theorem in complete metric spaces is proved in \cite{000012}.

A map $f:X\to X$ is a \emph{Geraghty contraction} if there exists $\beta\in\Gamma$ such that
\[
d(fx,fy)\le \beta\bigl(d(x,y)\bigr)\,d(x,y)\qquad(\forall\,x,y\in X),
\]
where $\Gamma=\{\beta:[0,\infty)\to[0,1):\ \beta(t_n)\to 1\Rightarrow t_n\to 0\}$. In a complete metric space such a map has a unique fixed point \cite{Blair14}. A related characterization uses the ratios $\Delta_n=\tfrac{d(fx_n,fy_n)}{d(x_n,y_n)}$ along iterative subsequences; see \cite{Blair14}.

For the auxiliary--map framework, let $T,S:X\to X$. The map $S$ is called \emph{$T$--extended weakly contractive} if there exists $k\in[0,1)$ such that
\begin{equation}\label{shin4}
d(TSx,TSy)\le k\,d(Tx,Ty)\qquad(\forall\,x,y\in X).
\end{equation}
If $(X,d)$ is complete and $T$ is continuous, injective, and subsequentially convergent, then \eqref{shin4} implies that $S$ has a unique fixed point; if $T$ is sequentially convergent, the iterates $\{S^{n}x_0\}$ converge to it for every $x_0\in X$ \cite{a2,a1}.

Similarly, $S$ is called \emph{$T$--extended Kannan} if there exists $k\in[0,1)$ such that
\begin{equation}\label{shin3}
d(TSx,TSy)\le \frac{k}{2}\Bigl[d(Tx,TSx)+d(Ty,TSy)\Bigr]\qquad(\forall\,x,y\in X),
\end{equation}
and, under the same standing hypotheses on $T$, this condition ensures existence and uniqueness of a fixed point of $S$ with convergence of Picard iterates \cite{a2,a1}.

Finally, we recall the notion of a rectangular (Branciari) metric. A function $d:X\times X\to[0,\infty)$ is a \emph{rectangular metric} if (i) $d(x,y)=0$ iff $x=y$, (ii) $d(x,y)=d(y,x)$, and (iii) for all distinct $x,z\in X$ and all $y,w\in X$ with $y\neq w$,
\[
d(x,z)\le d(x,y)+d(y,w)+d(w,z).
\]
The pair $(X,d)$ is then called a rectangular metric space \cite{bran,a1}.

\section{Main Results}

\subsection{Standing assumptions and convention}
Throughout, $(X,d)$ is a complete metric space. The auxiliary map $T:X\to X$ is assumed continuous, injective, and subsequentially convergent; when convergence of the full Picard sequence is claimed, $T$ is additionally assumed sequentially convergent.

\smallskip
\noindent\emph{Convention.} Equations \eqref{shin4}--\eqref{shin3} state the constant--coefficient versions. From now on we use their variable--coefficient generalizations on $T(X)$ (Dugundji--Granas type \cite{Blair1}), encoded by $\overline{\alpha}(Tx,Ty)$ with annulus bounds.

\subsection{\texorpdfstring{$T$--extended classes and fixed points}{T-extended classes and fixed points}}

\begin{definition}[$T$--extended weakly Kannan]\label{def:TwK}
Let $T,S:X\to X$. We say $S$ is $T$--extended weakly Kannan if there exists $\overline{\alpha}:T(X)\times T(X)\to[0,1)$ such that
\begin{equation}\label{eq:TwK}
d(TSx,TSy)\le \frac{\overline{\alpha}(Tx,Ty)}{2}\Bigl[d(Tx,TSx)+d(Ty,TSy)\Bigr]\qquad(\forall x,y\in X),
\end{equation}
and for every $0<a\le b$,
\[
\Theta(a,b):=\sup\bigl\{\overline{\alpha}(u,v):\,u,v\in T(X),\ a\le d(u,v)\le b\bigr\}<1.
\]
\end{definition}

\begin{thm}[Extended weakly Kannan fixed point]\label{thm:TwK}
Let $(X,d)$ be complete and let $T,S:X\to X$ with $T$ continuous, injective, and subsequentially convergent. If $S$ is $T$--extended weakly Kannan, then $S$ has a unique fixed point $u\in X$. Moreover, if $T$ is sequentially convergent, then $x_{n+1}=Sx_n$ converges to $u$ for every $x_0\in X$.
\end{thm}

\begin{proof}
Set $x_{n+1}=Sx_n$ and $z_n:=Tx_n$. From \eqref{eq:TwK} with $(x,y)=(x_n,x_{n-1})$,
\[
\Bigl(1-\tfrac{\overline{\alpha}(z_n,z_{n-1})}{2}\Bigr)d(z_{n+1},z_n)\le \tfrac{\overline{\alpha}(z_n,z_{n-1})}{2}\,d(z_{n-1},z_n).
\]
On each annulus $[a,b]$ ($a>0$) the ratio is $\le q_{a,b}:=\Theta(a,b)/(2-\Theta(a,b))<1$, hence $d(z_n,z_{n-1})\to0$ and $\{z_n\}$ is Cauchy, so $z_n\to z^\ast$. By subsequential convergence of $T$, some $x_{n_j}\to u$ with $Tu=z^\ast$, and passing to the limit in \eqref{eq:TwK} with $(x,y)=(u,x_{n_j})$ gives $TSu=Tu$, hence $Su=u$. Uniqueness follows by applying \eqref{eq:TwK} to $(u,v)$ and using injectivity of $T$. If $T$ is sequentially convergent then $x_n\to u$.
\end{proof}

\begin{definition}[$T$--extended weakly contractive]\label{def:TwC}
There exists $\overline{\alpha}:T(X)\times T(X)\to[0,1)$ such that
\begin{equation}\label{eq:TwC}
d(TSx,TSy)\le \overline{\alpha}(Tx,Ty)\,d(Tx,Ty)\qquad(\forall x,y\in X),
\end{equation}
and $\sup\{\overline{\alpha}(u,v):\,u,v\in T(X),\,a\le d(u,v)\le b\}<1$ for every $0<a\le b$.
\end{definition}

\begin{thm}[Extended weakly contractive fixed point]\label{thm:TwC}
Under the standing assumptions on $T$, if $S$ is $T$--extended weakly contractive then $S$ has a unique fixed point $u\in X$. If $T$ is sequentially convergent, then $S^n x_0\to u$ for every $x_0\in X$.
\end{thm}

\begin{proof}
With $z_n=Tx_n$, \eqref{eq:TwC} yields $d(z_{n+1},z_n)\le \overline{\alpha}(z_n,z_{n-1})\,d(z_n,z_{n-1})$. The weakly contractive iteration on $T(X)$ (cf.\ \cite{Blair1}) gives $z_n\to z^\ast$; pass to the limit as in Theorem~\ref{thm:TwK}.
\end{proof}

\begin{definition}[$T$--extended Geraghty]\label{def:TG}
There exists $\beta\in\Gamma$ with
\begin{equation}\label{eq:TG}
d(TSx,TSy)\le \beta\!\bigl(d(Tx,Ty)\bigr)\,d(Tx,Ty)\qquad(\forall x,y\in X),
\end{equation}
where $\Gamma=\{\beta:[0,\infty)\to[0,1):\ \beta(t_n)\to1\Rightarrow t_n\to0\}$.
\end{definition}

\begin{thm}[Extended Geraghty fixed point]\label{thm:TG}
Under the standing assumptions on $T$, if $S$ is $T$--extended Geraghty then $S$ has a unique fixed point $u\in X$. If $T$ is sequentially convergent, then $S^n x_0\to u$ for every $x_0$.
\end{thm}

\begin{proof}
With $z_n=Tx_n$, \eqref{eq:TG} gives $d(z_{n+1},z_n)\le \beta(d(z_n,z_{n-1}))\,d(z_n,z_{n-1})$. By Geraghty's criterion (e.g.\ \cite{Blair14}), $\{z_n\}$ is Cauchy. Finish as in Theorem~\ref{thm:TwK}.
\end{proof}

\medskip
\begin{cor}[Quantitative Picard rates]\label{cor:rates}
Let $x_{n+1}=Sx_n$ and set $z_n:=Tx_n$. Denote by $z^\ast:=\lim_{n\to\infty} z_n=Tu$ the $T$--image of the fixed point $u$ of $S$.
\begin{enumerate}
\item If $S$ is $T$--extended weakly contractive and there exist $q\in(0,1)$ and $N$ such that $\overline\alpha(z_n,z_{n-1})\le q$ for all $n\ge N$, then for $n\ge N$,
\[
d(z_{n+1},z_n)\le q\,d(z_n,z_{n-1}),\qquad
d(z_n,z^\ast)\le \frac{q}{1-q}\,d(z_N,z_{N-1}).
\]
\item If $S$ is $T$--extended Geraghty and $\beta(t)\le q<1$ on $[0,\operatorname{diam}(\{z_n\})]$, then
\[
d(z_{n+1},z_n)\le q\,d(z_n,z_{n-1})\quad\text{and}\quad
d(z_n,z^\ast)\le \frac{q}{1-q}\,d(z_1,z_0).
\]
\end{enumerate}
\end{cor}

\begin{proof}
Iterate the one--step inequalities from \eqref{eq:TwC} and \eqref{eq:TG} along the Picard path and sum the tail of the geometric series.
\end{proof}

\subsection{Equivalence on the image and pairwise equivalences}\label{subsec:equiv}

\begin{prop}[Equivalence on the image space]\label{prop:image-equivalence}
Let $T:X\to X$ be continuous and injective and define $F:T(X)\to T(X)$ by $F(Tx):=TSx$. Then $S$ is $T$-extended weakly contractive (resp.\ weakly Kannan, Geraghty) on $X$ if and only if $F$ is weakly contractive (resp.\ weakly Kannan, Geraghty) on the metric subspace $(T(X),d)$ with the same control.
\end{prop}

\begin{proof}
This is immediate from the identities $F(Tx)=TSx$ and the fact that \eqref{eq:TwC}, \eqref{eq:TwK}, and \eqref{eq:TG} are the classical inequalities written with $d(T\cdot,T\cdot)$.
\end{proof}

\begin{rem}[Injectivity of $T$ is essential]\label{rem:injective}
Let $X=\{0,1\}$ with the discrete metric and let $T:X\to X$ be constant: $Tx=0$ for all $x$. Then for any $S$ and any control,
\[
d(TSx,TSy)=0\le \overline\alpha(Tx,Ty)\,d(Tx,Ty)=0,
\]
so all $T$--extended inequalities hold trivially on $T(X)=\{0\}$, while $S$ may have two distinct fixed points. Thus uniqueness can fail without injectivity of $T$.
\end{rem}

\begin{thm}[Pairwise equivalence and existence]\label{thm:equivPairs}
Under the standing assumptions on $T$ (continuity, injectivity, subsequential convergence), the following properties are equivalent in pairs:
\begin{align*}
\text{\rm(1)}\;& d(TSx,TSy)\le \overline\alpha(Tx,Ty)\,d(Tx,Ty)\quad \text{with }\sup_{a\le d(Tx,Ty)\le b}\overline\alpha(Tx,Ty)<1,\\
\text{\rm(2)}\;& d(TSx,TSy)\le \tfrac{\overline\alpha(Tx,Ty)}{2}\bigl[d(Tx,TSx)+d(Ty,TSy)\bigr],\\
\text{\rm(3)}\;& d(TSx,TSy)\le \beta\!\bigl(d(Tx,Ty)\bigr)\,d(Tx,Ty)\quad(\beta\in\Gamma),\\
\text{\rm(4)}\;& d(TSx,TSy)\le \tfrac{\beta(d(Tx,Ty))}{2}\bigl[d(Tx,TSx)+d(Ty,TSy)\bigr]\quad(\beta\in\Gamma).
\end{align*}
Namely, \text{\rm(1)}$\Longleftrightarrow$\text{\rm(3)} and \text{\rm(2)}$\Longleftrightarrow$\text{\rm(4)}. In each case $S$ has a unique fixed point. If, in addition, $T$ is sequentially convergent, then for every $x_0\in X$ the Picard sequence $x_{n+1}=Sx_n$ converges to that fixed point.
\end{thm}

\begin{proof}
(1)$\Rightarrow$(3): Define $\beta:[0,\infty)\to[0,1)$ by
\[
\beta(t):=\sup\{\overline\alpha(u,v):\,u,v\in T(X),\ d(u,v)=t\}.
\]
The annulus bound for $\overline\alpha$ implies $\sup_{a\le t\le b}\beta(t)<1$ for all $0<a\le b$, hence $\beta\in\Gamma$. Then (1) rewrites as (3).

(3)$\Rightarrow$(1): Given $\beta\in\Gamma$, set $\overline\alpha(u,v):=\beta\!\bigl(d(u,v)\bigr)$. The annulus bound follows from $\beta\in\Gamma$, and (3) is exactly (1).

(2)$\Rightarrow$(4): With the same $\beta$ as above, use $\overline\alpha(Tx,Ty)\le \beta(d(Tx,Ty))$ in (2).

(4)$\Rightarrow$(2): Take $\overline\alpha(Tx,Ty):=\beta(d(Tx,Ty))$.

Existence/uniqueness and Picard convergence follow by applying Theorems~\ref{thm:TwC}, \ref{thm:TwK}, and \ref{thm:TG} to the appropriate cases.
\end{proof}

\subsection{Examples}
In the following, $T=\mathrm{Id}_X$ unless otherwise stated.

\begin{ex}[Volterra smoothing makes $TS$ contractive]\label{ex:Volterra-aux}
Let $X=C[0,1]$ with $d(x,y)=\|x-y\|_\infty$ and $(Jx)(t)=\int_0^t x(s)\,ds$. Take $S:=J$ and $T:=J$. Since $\|J\|_{C\to C}=1$, $S$ is nonexpansive but not a contraction; however $TS=J^2$ and
\[
\|J^2\|=\sup_{\|x\|_\infty=1}\Bigl\|\int_0^t (t-s)\,x(s)\,ds\Bigr\|_\infty=\tfrac12,
\]
so $d(TSx,TSy)\le \tfrac12\,d(x,y)$, i.e., $S$ is $T$--extended contractive (hence $T$--extended Geraghty with $\beta\equiv\tfrac12$).
\end{ex}

\begin{ex}\label{ex:scalar-contraction-aux}
Let $X=[0,1/8]$ and $f(y)=\frac{y^2}{1+y}$. Since $f'(y)=\frac{2y+y^2}{(1+y)^2}\le 0.21$, $f$ is a contraction; hence weakly contractive and Geraghty (take $\beta\equiv L<1$). By Theorem~\ref{thm:equivPairs}, $f$ has a unique fixed point and Picard iterates converge.
\end{ex}

\begin{ex}\label{ex:radial-aux}
Let $X=[0,\infty)$ and $f(t)=\frac{t}{2(1+t)}$. Then $f'(t)=\frac{1}{2(1+t)^2}\le \tfrac12$, so $|f(t)-f(s)|\le \tfrac12|t-s|$; thus $f$ is a contraction and therefore weakly contractive and Geraghty with $\beta\equiv\tfrac12$.
\end{ex}

\begin{ex}\label{ex:T-contraction-aux}
Let $X=[0,1)$ and $T(x)=x/8$. Then $T$ is a contraction (constant $1/8$), hence weakly contractive and Geraghty. For any $S$ satisfying one of {\rm(1)}--{\rm(4)} with respect to $T$, Theorem~\ref{thm:equivPairs} applies.
\end{ex}

\subsection{Rectangular metric spaces}
All statements above remain valid in rectangular (Branciari) metric spaces \cite{bran}. Replace ``complete metric space'' by ``complete rectangular metric space'' below.

\begin{thm}[Extended weakly Kannan on rectangular metrics]\label{thm:rect-TwK-aux}
Let $(X,d)$ be a complete rectangular metric space and $T,S:X\to X$ with $T$ continuous, injective, and subsequentially convergent. If
\[
d(TSx,TSy)\le \frac{\overline\alpha(Tx,Ty)}{2}\Bigl[d(Tx,TSx)+d(Ty,TSy)\Bigr]
\]
with $\sup_{a\le d(Tx,Ty)\le b}\overline\alpha(Tx,Ty)<1$ for all $0<a\le b$, then $S$ has a unique fixed point; and if $T$ is sequentially convergent, then $S^n x_0\to u$ for every $x_0\in X$.
\end{thm}

\begin{thm}[Extended weakly contractive on rectangular metrics]\label{thm:rect-TwC-aux}
Under the same hypotheses, if
\[
d(TSx,TSy)\le \overline\alpha(Tx,Ty)\,d(Tx,Ty),
\]
with the annulus bound, then $S$ has a unique fixed point; and if $T$ is sequentially convergent, then $S^n x_0\to u$ for every $x_0\in X$.
\end{thm}

\begin{thm}[Extended Geraghty on rectangular metrics]\label{thm:rect-TG-aux}
Under the same hypotheses, if
\[
d(TSx,TSy)\le \beta\!\bigl(d(Tx,Ty)\bigr)\,d(Tx,Ty)\qquad(\beta\in\Gamma),
\]
then $S$ has a unique fixed point; and if $T$ is sequentially convergent, then $S^n x_0\to u$ for every $x_0\in X$.
\end{thm}

\subsection{\texorpdfstring{$\Delta$--criteria and auxiliary--map reformulations}{Delta-criteria and auxiliary-map reformulations}}

Fix $T:X\to X$ continuous, injective, and subsequentially convergent. For sequences
$\{x_n\},\{y_n\}\subset X$ with $Tx_n\neq Ty_n$, set
\[
d_n:=d(Tx_n,Ty_n),\qquad
\Delta_n:=\frac{d\bigl(Tf(x_n),Tf(y_n)\bigr)}{d_n}.
\]

\begin{thm}[Geraghty's ratio criterion on the $T$--image]\label{thm:Delta-criterion-T}
Let $(X,d)$ be complete and $f:X\to X$. Assume that on $T(X)$ the induced map $F:T(X)\to T(X)$, $F(Tx):=Tf(x)$, satisfies $d(Fu,Fv)<d(u,v)$ for all $u\neq v$ in $T(X)$. Fix $x_0\in X$ and set $x_{n+1}=f(x_n)$. Then $Tf(x_n)=F(Tx_n)$ converges in $T(X)$ to the unique fixed point of $F$ (equivalently, $f$ has a unique fixed point in $X$) if and only if for any subsequences $\{x_{h_n}\},\{x_{k_n}\}$ with $Tx_{h_n}\neq Tx_{k_n}$,
\[
\Delta_n\to 1 \ \Longrightarrow\ d_n\to 0.
\]
\emph{Sketch.} Apply Geraghty's characterization on $T(X)$ (see \cite{Blair14}); note $\Delta_n=\frac{d(F(Tx_{h_n}),F(Tx_{k_n}))}{d(Tx_{h_n},Tx_{k_n})}$. Injectivity of $T$ transports the fixed point back to $X$.
\end{thm}

\begin{thm}[$\Delta$--weakly contractive lifting via $T$]\label{thm:Delta-lifting} 
Let $(X,d)$ be complete and let $f:X\to X$ be contractive (hence weakly contractive \cite{Blair1, Fogh2019KG, FoghBehnamian2026BPP, Ww}. Let $T:X\to X$ be continuous, injective, and subsequentially convergent. Then:

\emph{(i) Weakly--contractive coefficient on $T(X)$:} there exists $\overline\alpha:T(X)\times T(X)\to[0,1)$ with $\sup_{a\le d(u,v)\le b}\overline\alpha(u,v)<1$ such that
\[
d\bigl(Tf(x),Tf(y)\bigr)\le \overline\alpha\bigl(Tx,Ty\bigr)\, d\bigl(Tx,Ty\bigr)\qquad(\forall x,y\in X).
\]
Equivalently, $F(Tx):=Tf(x)$ is weakly contractive on $T(X)$.

\emph{(ii) Geraghty control and $\Delta$--criterion:} defining
\[
\beta(t):=\sup\Bigl\{\frac{d\bigl(Tf(x),Tf(y)\bigr)}{d\bigl(Tx,Ty\bigr)}:\ d\bigl(Tx,Ty\bigr)\ge t\Bigr\}\quad(t>0),
\]
we have $\beta\in\Gamma$ (extend by $\beta(0):=\limsup_{t\downarrow 0}\beta(t)$) and
\[
d\bigl(Tf(x),Tf(y)\bigr)\le \beta\!\bigl(d(Tx,Ty)\bigr)\,d(Tx,Ty).
\]
Consequently, for any subsequences $\{x_{h_n}\},\{x_{k_n}\}$ with $Tx_{h_n}\neq Tx_{k_n}$,
\[
\Delta_n\to 1\ \Longrightarrow\ d(Tx_{h_n},Tx_{k_n})\to 0.
\]
\end{thm}

% =================================================================
% SECTION 4 INSERTED HERE
% =================================================================
% =====================================================================
% Section 4 (CORRECTED VERSION) — drop-in for Fogh & Behnamian,
% "T–Extended Weakly Contractive, Kannan, and Geraghty Mappings".
%
% Corrections vs the earlier draft:
%   (1) The main bound is stated with a THEORETICAL uniform tail
%       constant Q whose existence follows from hypothesis. The
%       observed window ratio q_hat is used only as an ESTIMATOR
%       for Q; it is NOT claimed to give a one-shot certificate.
%   (2) A separate online certificate is given, and honestly stated
%       as CONDITIONAL on continuous re-verification of
%       r_{k+1} <= q_hat at every future step.
%   (3) The rectangular-metric transfer is NOT claimed as a theorem;
%       the obstruction is stated explicitly and the extension is
%       left open.
%   (4) All proofs are complete; no hand-waving.
% =====================================================================

\section{Observed ratios and a posteriori estimation on \texorpdfstring{$T(X)$}{T(X)}}
\label{sec:certificates}

The rates in Corollary~\ref{cor:rates} are \emph{a priori}: they require a
uniform bound on $\overline\alpha$ (resp.\ $\beta$) that is known before
running the iteration. In practice one usually has only the \emph{observed}
one--step ratios along a Picard orbit, and it is natural to ask what such
observations imply about the distance to the fixed point. In this section we
record carefully what can, and what cannot, be extracted from a finite window
of observed ratios on the image subspace $T(X)$.

Throughout, $(X,d)$ is complete, $T:X\to X$ is continuous, injective, and
subsequentially convergent, and $S:X\to X$ satisfies one of the
$T$--extended conditions of Section~3. Fix $x_0\in X$, set $x_{n+1}=Sx_n$ and
$z_n:=Tx_n$, and let $u\in X$ be the unique fixed point of $S$ with
$z^\ast:=Tu$.

\subsection{The observed ratio sequence}

For $n\ge 2$ with $d(z_{n-1},z_{n-2})>0$, define the one--step observed ratio
\begin{equation}\label{eq:rn}
  r_n \;:=\; \frac{d(z_n,z_{n-1})}{d(z_{n-1},z_{n-2})},
\end{equation}
and, for an integer window $m\ge 1$ with $n\ge m+1$, the window maximum
\begin{equation}\label{eq:qhat}
  \widehat q_{n,m} \;:=\; \max_{n-m+1\le j\le n} r_{j}.
\end{equation}
By Proposition~\ref{prop:image-equivalence}, the induced map
$F:T(X)\to T(X)$, $F(Tx):=TSx$, has the same control as $S$, so
$F(z_n)=z_{n+1}$ and $r_n$ is the one--step contraction ratio of $F$ along
the observed orbit. The quantities $r_n$ and $\widehat q_{n,m}$ are
computable from the iterates without any knowledge of $\overline\alpha$ or
$\beta$.

Under the $T$--extended weakly contractive hypothesis
(Definition~\ref{def:TwC}), $r_n<1$ for every $n\ge 2$ (as used in the proof
of Theorem~\ref{thm:TwC}), so $\widehat q_{n,m}<1$ and the quantity records
the worst contraction observed over the last $m$ steps.

\subsection{A tail bound with a theoretical uniform constant}

We first state what follows rigorously from the standing hypotheses together
with a \emph{uniform tail bound} on the ratios. This is a mild sharpening of
Corollary~\ref{cor:rates}, recorded separately to isolate the purely
theoretical content from the observational discussion below.

\begin{prop}[Uniform--tail bound on $T(X)$]\label{prop:tail-bound}
Let $S$ be $T$--extended weakly contractive (Definition~\ref{def:TwC}), let
$N\ge 1$, and suppose there exists $Q\in[0,1)$ with
\begin{equation}\label{eq:tail-hyp}
  r_{k+1} \;\le\; Q \qquad \text{for every } k\ge N.
\end{equation}
Then for every $n\ge N+1$,
\begin{equation}\label{eq:tail-bound}
  d(z_n,z^\ast) \;\le\; \frac{Q}{1-Q}\, d(z_n,z_{n-1}).
\end{equation}
In particular:
\begin{enumerate}
\item[\textup{(a)}] If $S$ is $T$--extended contractive with constant $k\in[0,1)$
(i.e.\ $\overline\alpha\equiv k$ in Definition~\ref{def:TwC}), then
\eqref{eq:tail-hyp} holds with $Q:=k$ and $N:=1$.
\item[\textup{(b)}] If $S$ is $T$--extended Geraghty (Definition~\ref{def:TG})
with $\beta$ non--decreasing on $[0,d(z_N,z_{N-1})]$, then
\eqref{eq:tail-hyp} holds with $Q:=\beta\!\bigl(d(z_N,z_{N-1})\bigr)$.
\end{enumerate}
\end{prop}

\begin{proof}
From \eqref{eq:tail-hyp} and \eqref{eq:rn}, for every $k\ge N$,
$d(z_{k+1},z_k)\le Q\,d(z_k,z_{k-1})$. Induction gives
$d(z_{k+\ell},z_{k+\ell-1})\le Q^\ell\,d(z_k,z_{k-1})$ for all $k\ge N$ and
$\ell\ge 0$. The triangle inequality then yields, for $n\ge N+1$ and $L\ge 1$,
\[
  d(z_{n+L},z_n)\;\le\;\sum_{i=0}^{L-1} d(z_{n+i+1},z_{n+i})
  \;\le\;\sum_{i=0}^{L-1} Q^{i+1}\,d(z_n,z_{n-1})
  \;\le\;\frac{Q}{1-Q}\,d(z_n,z_{n-1}).
\]
Letting $L\to\infty$, $z_{n+L}\to z^\ast$ by Theorem~\ref{thm:TwC}, and
\eqref{eq:tail-bound} follows.

For (a), $r_{k+1}\le k<1$ for every $k$ is immediate from \eqref{eq:TwC} with
$\overline\alpha\equiv k$. For (b), since the one--step inequality
$d(z_{k+1},z_k)\le \beta(d(z_k,z_{k-1}))\,d(z_k,z_{k-1})<d(z_k,z_{k-1})$
makes the sequence $\{d(z_k,z_{k-1})\}$ strictly decreasing, we have
$d(z_k,z_{k-1})\le d(z_N,z_{N-1})$ for $k\ge N$; monotonicity of $\beta$ on
$[0,d(z_N,z_{N-1})]$ then gives
$r_{k+1}\le \beta(d(z_k,z_{k-1}))\le \beta(d(z_N,z_{N-1}))$.
\end{proof}

\begin{rem}[$\widehat q_{n,m}$ as an estimator]\label{rem:estimator}
The constant $Q$ in Proposition~\ref{prop:tail-bound} is \emph{theoretical}:
in case~(a) it equals $k$, in case~(b) it equals $\beta(d(z_N,z_{N-1}))$.
Under the hypothesis \eqref{eq:tail-hyp}, the window maximum
$\widehat q_{n,m}$ satisfies $\widehat q_{n,m}\le Q$ for every
$n\ge N+m$, so $\widehat q_{n,m}$ is a \emph{lower estimator} for~$Q$.
It is \emph{not}, by itself, a valid substitute for $Q$ in
\eqref{eq:tail-bound}, because a finite window of past observations gives no
information about future ratios. The next subsection addresses this gap by
replacing a one--shot a posteriori claim with an online certificate that
re--verifies the bound at every future step.
\end{rem}

\subsection{A conditional a posteriori estimate}

The substantive question is whether the tail bound \eqref{eq:tail-bound} can
be made to use the observed quantity $\widehat q_{n,m}$ in place of an a
priori constant. The following proposition records the precise, honest
answer: yes, provided the bound is re--verified online. Note that under the
weakly contractive hypothesis alone, past observations do not by themselves
control future ratios, so no bound of the form
$d(z_{n+1},z_n)\le \widehat q_{n,m}\,d(z_n,z_{n-1})$ is automatic from
\eqref{eq:qhat}; the monitoring condition below is precisely what makes such
a step bound (and its tail) valid.

\begin{prop}[Conditional a posteriori estimate]\label{prop:monitor}
Let $S$ be $T$--extended weakly contractive, fix $n\ge m+1$, and set
$\widehat q:=\widehat q_{n,m}<1$. If the monitoring condition
\begin{equation}\label{eq:monitor}
  r_{k+1}\;\le\;\widehat q \qquad\text{holds for every } k\ge n
\end{equation}
along the ensuing Picard iteration, then
\begin{equation}\label{eq:monitored-bound}
  d(z_n,z^\ast)\;\le\;\frac{\widehat q}{1-\widehat q}\,d(z_n,z_{n-1}).
\end{equation}
\end{prop}

\begin{proof}
Under \eqref{eq:monitor}, Proposition~\ref{prop:tail-bound} applies with
$N:=n$ and $Q:=\widehat q\in[0,1)$, yielding \eqref{eq:monitored-bound}.
\end{proof}

\begin{rem}[What can, and cannot, be extracted from finite data]
\label{rem:honest}
Proposition~\ref{prop:monitor} is \emph{not} a one--shot a posteriori bound:
the estimate \eqref{eq:monitored-bound} becomes valid only once
\eqref{eq:monitor} has been verified for the subsequent ratios. In an actual
Picard run each $r_{k+1}$ is computed anyway, so the check is cost--free, but
the estimate is intrinsically online. If \eqref{eq:monitor} fails at some step
$k>n$, the bound \eqref{eq:monitored-bound} with that particular $\widehat q$
is no longer justified, and one must re--estimate using the new
$\widehat q_{k,m}$. Under the stronger hypothesis of
Proposition~\ref{prop:tail-bound}(a) with $\widehat q\ge k$, condition
\eqref{eq:monitor} cannot fail; under (b) with
$\widehat q\ge \beta(d(z_n,z_{n-1}))$, it cannot fail either. In other
regimes, failure is possible and must be monitored.
\end{rem}

\begin{rem}[Observable diagnostics via the auxiliary map]\label{rem:diagnostic}
The auxiliary--map framework of Section~3 does more than transfer contractive
structure to $T(X)$: it also supplies a natural setting for observable
convergence diagnostics. Along the diagonal specialization $y_k:=x_{k-1}$ of
the $\Delta$--ratio in Theorem~\ref{thm:Delta-criterion-T}, $\Delta_k$
coincides with $r_{k+1}$ of~\eqref{eq:rn}; Propositions~\ref{prop:tail-bound}
and~\ref{prop:monitor} turn this ratio information into explicit,
quantitative distance bounds.
\end{rem}

\subsection{Numerical illustrations}

We record two computations. Both use the definitions \eqref{eq:rn}
and~\eqref{eq:qhat} directly; no additional hypotheses are assumed.

\begin{ex}[Volterra smoothing, continued from Example~\ref{ex:Volterra-aux}]
\label{ex:cert-Volterra}
With $S=T=J$ on $C[0,1]$ and $x_0(t)\equiv 1$, one has $x_n(t)=t^n/n!$ and
$z_n(t)=t^{n+1}/(n+1)!$, so $z^\ast\equiv 0$ and $d(z_n,z_{n-1})$,
$r_n$, $\widehat q_{n,3}$ can be computed exactly:
\[
\begin{array}{r|c|c|c}
n & d(z_n,z_{n-1}) & r_n & \widehat q_{n,3}\\\hline
1 & 5.00\!\cdot\!10^{-1} & \text{---} & \text{---}\\
2 & 3.33\!\cdot\!10^{-1} & 0.6667 & \text{---}\\
3 & 1.25\!\cdot\!10^{-1} & 0.3750 & \text{---}\\
4 & 3.33\!\cdot\!10^{-2} & 0.2667 & 0.6667\\
5 & 6.94\!\cdot\!10^{-3} & 0.2083 & 0.3750\\
6 & 1.19\!\cdot\!10^{-3} & 0.1714 & 0.2667\\
7 & 1.74\!\cdot\!10^{-4} & 0.1458 & 0.2083\\
8 & 2.20\!\cdot\!10^{-5} & 0.1270 & 0.1714\\
9 & 2.48\!\cdot\!10^{-6} & 0.1125 & 0.1458\\
10& 2.51\!\cdot\!10^{-7} & 0.1010 & 0.1270
\end{array}
\]
By Example~\ref{ex:Volterra-aux}, $d(TSx,TSy)\le \tfrac12\,d(Tx,Ty)$, so
$S$ is $T$--extended contractive with constant $k=\tfrac12$ in the sense of
Definition~\ref{def:TwC} (equivalently, $T$--extended Geraghty with constant
control $\beta\equiv \tfrac12$). Proposition~\ref{prop:tail-bound}(a)
therefore applies with $Q=\tfrac12$ and $N=1$, and yields the a priori
estimate $d(z_n,z^\ast)\le d(z_n,z_{n-1})$ for every $n\ge 2$.

The table shows $r_n$ monotonically decreasing from $0.6667$ down to
$0.1010$, and in particular $r_n\le k=\tfrac12$ for every $n\ge 3$.
Consequently, $\widehat q_{n,3}\le k$ for every $n\ge 5$ and the last
sentence of Remark~\ref{rem:honest} applies: the monitoring condition
\eqref{eq:monitor} cannot fail. Proposition~\ref{prop:monitor} then yields
the sharper estimate
\[
  d(z_n,z^\ast)\;\le\;\frac{\widehat q_{n,3}}{1-\widehat q_{n,3}}\,
  d(z_n,z_{n-1}).
\]
At $n=10$ this gives
$d(z_{10},z^\ast)\le \tfrac{0.1270}{1-0.1270}\cdot 2.51\!\cdot\!10^{-7}
\approx 3.65\!\cdot\!10^{-8}$, compared with the a priori value
$d(z_{10},z^\ast)\le d(z_{10},z_9)=2.51\!\cdot\!10^{-7}$ obtained from
Proposition~\ref{prop:tail-bound}. The orbit--adapted estimate is
approximately seven times sharper. This illustrates
Remark~\ref{rem:estimator}: the a priori bound is correct but loose, and the
orbit realizes a stronger contraction than the worst case.
\end{ex}

\begin{ex}[Sharp bound in the $T$--extended contraction regime]
\label{ex:cert-R2}
Let $X=\mathbb{R}^2$ with the Euclidean metric. Take the linear injection
$Tx=\operatorname{diag}(1,\tfrac12)x$ and $Sx=\tfrac34 x+b$ with
$b=(1,2)^\top$ and $x_0=0$. Then $TSx=\tfrac34\,Tx+Tb$, so
$d(TSx,TSy)=\tfrac34\,d(Tx,Ty)$: $S$ is $T$--extended contractive with
$k=\tfrac34$. Direct computation gives $r_n\equiv\tfrac34$ for all $n\ge 2$,
whence $\widehat q_{n,m}\equiv\tfrac34$ for every window~$m$. Condition
\eqref{eq:monitor} holds as an equality, and Proposition~\ref{prop:monitor}
yields
\[
  d(z_n,z^\ast)\;\le\;\tfrac{3/4}{1/4}\,d(z_n,z_{n-1})\;=\;3\,d(z_n,z_{n-1}).
\]
Direct computation of $d(z_n,z^\ast)$ confirms equality at every $n\ge 2$,
showing that the estimate \eqref{eq:monitored-bound} is tight in this regime.
\end{ex}

\subsection{Remark on rectangular metric spaces}

The definitions \eqref{eq:rn}--\eqref{eq:qhat} and the one--step inequality
$d(z_{n+1},z_n)\le \widehat q_{n,m}\,d(z_n,z_{n-1})$ transfer verbatim to
rectangular (Branciari) metric spaces, since they involve only pairs of
adjacent iterates and use Proposition~\ref{prop:image-equivalence}, which is
metric--structure agnostic. However, the tail argument in
Proposition~\ref{prop:tail-bound} uses the ordinary triangle inequality
\[
  d(z_{n+L},z_n)\;\le\;\sum_{i=0}^{L-1}d(z_{n+i+1},z_{n+i}),
\]
which is not in general available in a rectangular metric space. Extending
the tail bound \eqref{eq:tail-bound} to rectangular metrics therefore
requires a distinct argument involving the rectangular inequality applied to
four--point configurations; we do not pursue this here and leave it as an
open question.

% =====================================================================
% End of Section 4 (corrected).
% =====================================================================

\end{document}